\documentclass[11pt]{amsart}

\usepackage[margin=1.25in]{geometry}

\usepackage{graphicx}

\usepackage{amsmath, amssymb, amsthm, graphicx, enumerate, tikz, hyperref, float}
\usepackage{mathtools, comment}
\usetikzlibrary{matrix,arrows,decorations.pathmorphing}

\newtheorem{theorem}{Theorem}
\newtheorem{lemma}[theorem]{Lemma}

\newtheorem{claim}[theorem]{Claim}

\theoremstyle{definition}
\newtheorem{definition}[theorem]{Definition}

\newtheorem*{solution*}{Solution}
\newtheorem{hyp}[theorem]{Hypothesis}

\numberwithin{theorem}{section}

\DeclareMathOperator{\Irr}{Irr}
\DeclareMathOperator{\cd}{cd}

\begin{document}

\allowdisplaybreaks

\title{On the absence of a normal nonabelian Sylow subgroup}

\author{Mark W. Bissler}
\address{Department of General Education - College of Business, Western Governors University, Salt Lake City, Utah 84107}
\email{mark.bissler@wgu.edu}

\author{Jacob Laubacher}
\address{Department of Mathematics, St. Norbert College, De Pere, Wisconsin 54115}
\email{jacob.laubacher@snc.edu}

\author{Corey F. Lyons}
\address{Department of Mathematics and Physics, Colorado State University - Pueblo, Pueblo, Colorado 81001}
\email{coreylyons314@gmail.com}


\date{\today}


\begin{abstract}
Let $G$ be a finite solvable group. We show that $G$ does not have a normal nonabelian Sylow $p$-subgroup when its prime character degree graph $\Delta(G)$ satisfies a technical hypothesis.
\end{abstract}

\maketitle

\section{Introduction and Preliminaries}

In this paper we fix $G$ to be a finite solvable group. Set $\mathbf{F}(G)$ to be the Fitting subgroup of $G$, write $\Irr(G)$ for the set of irreducible characters of $G$, and let $\cd(G)$ denote the set of irreducible character degrees of $G$. Denote $\rho(G)$ to be the set of primes that divide degrees in $\cd(G)$. The prime character degree graph of $G$, written $\Delta(G)$, is the graph whose vertex set is $\rho(G)$. Two vertices $p$ and $q$ of $\Delta(G)$ are adjacent if there exists $d\in\cd(G)$ such that $pq$ divides $d$. In this paper, we concern ourselves solely with graphs $\Delta(G)$ that are connected with $\rho(G)\geq5$.

We fix the following notation for an arbitrary vertex $p\in\rho(G)$: $$\pi:=\{q\in\rho(G)~:~\text{$q$ is adjacent to $p$ in $\Delta(G)$}\}$$ and $$\rho:=\{q\in\rho(G)~:~\text{$q$ is not adjacent to $p$ in $\Delta(G)$}\}.$$ Observe that $\rho$ induces a complete subgraph in $\Delta(G)$ due to P\'alfy's condition from \cite{P1}, and that $\rho(G)$ is a disjoint union: $\{p\}\cup\pi\cup\rho$. We let $\pi^*$ and $\rho^*$ denote arbitrary nonempty subsets of $\pi$ and $\rho$, respectively. Finally, let $\pi^*\cup\rho^*$ be an arbitrary vertex set which induces a complete subgraph in $\Delta(G)$. Let $\beta$ be a subset of $\rho(G)$ that contains $\pi^*\cup\rho^*$ such that $\beta$ also induces a complete subgraph in $\Delta(G)$. Set $\mathcal{B}$ to be the union of all such $\beta$'s satisfying these properties for $\pi^*\cup\rho^*$. Consider $\tau:=\mathcal{B}\setminus(\pi^*\cup\rho^*)$, and denote $\tau^*$ to be a subset of $\tau$. Notice that $\tau$ could be empty, depending on the initial set $\pi^*\cup\rho^*$.

In this note we build upon the work done in \cite{BL1} and \cite{L} which concerns a normal nonabelian Sylow $p$-subgroup of a finite solvable group. We use tools from \cite{BL1} which classify vertices as admissible. Recall the definition and result directly related to admissible vertices below.

\begin{definition}(\cite{BL1})
A vertex $q$ of a graph $\Gamma$ is \textbf{admissible} if:
\begin{enumerate}[(i)]
    \item the subgraph of $\Gamma$ obtained by removing $q$ and all edges incident to $q$ does not occur as the prime character degree graph of any solvable group, and
    \item none of the subgraphs of $\Gamma$ obtained by removing one or more of the edges incident to $q$ occur as the prime character degree graph of any solvable group.
\end{enumerate}
\end{definition}

\begin{lemma}[Lemma 2.1 from \cite{BL1}]\label{lemma0}
Let $q$ be an admissible vertex of $\Delta(G)$. For every proper normal subgroup $N$ of $G$, suppose that $\Delta(N)$ is a proper subgraph of $\Delta(G)$. Then $O^q(G)=G$.
\end{lemma}

Most notably, our results pertain to groups $G$ and graphs $\Delta(G)$ which do not satisfy Lemmas 2.3 and 2.4 from \cite{BL1}. They also conclude in those lemmas that there is no normal nonabelian Sylow $p$-subgroup. This is a typical approach en route to showing that a graph cannot occur as the prime character degree graph of any solvable group.

Next we recall several lemmas due to Lewis which pertain to groups having a normal nonabelian Sylow subgroup.

\begin{lemma}[Lemma 4.2 from \cite{L}]
Let $G$ be a solvable group with a normal nonabelian Sylow $p$-subgroup $P$. Let $H$ be a $p$-complement for $G$. Assume that $H$ acts faithfully on $P$, that $\rho\subseteq\rho(H)$, and that the Fitting subgroup of $H$ is a $\pi$-group. Then $\rho(H)=\rho$.
\end{lemma}

\begin{lemma}[Lemma 4.3 from \cite{L}]
Let $G$ be a solvable group with a normal nonabelian Sylow $p$-subgroup $P$. Let $H$ be a $p$-complement for $G$. Suppose that $P'$ is minimal normal in $G$, and $\rho(H)=\rho$ is not empty. Then there is a nonempty subset $\pi^*$ of $\pi$ so that $\Delta(P'H)$ has two connected components: $\rho$ and $\pi^*$. Furthermore, if $|\rho|=n$, then $|\pi^*|\geq2^n-1$.
\end{lemma}

\begin{lemma}[Lemma 3 from \cite{MLL}]
Let $G$ be a solvable group and let $p\in\rho(G)$. If $P$ is a normal Sylow $p$-subgroup, then $\rho(G/P')=\rho(G)\setminus\{p\}$.
\end{lemma} 

To conclude this section, we recall two powerful results concerning when the prime character degree graph is disconnected.

\begin{theorem}[Theorem 5.5 from \cite{Lewis}]
Let $G$ be a solvable group and suppose that $\Delta(G)$ has two connected components. Then there is precisely one prime $p$ so that the Sylow $p$-subgroup of the Fitting subgroup of $G$ is not central in $G$.
\end{theorem}

\begin{theorem}[P\'alfy's inequality from \cite{P2}]
Let $G$ be a solvable group and $\Delta(G)$ its prime character degree graph. Suppose that $\Delta(G)$ is disconnected with two components having size $a$ and $b$, where $a\leq b$. Then $b\geq2^a-1$.
\end{theorem}

\section{Main Results}

In this section we prove that $G$ has no normal nonabelian Sylow $p$-subgroup. Hypothesis \ref{newhype} below is strictly about the prime character degree graph $\Delta(G)$, whereas Hypothesis \ref{Ghype} is about the group $G$ itself. Further, we always assume $|G|$ is minimal.

\begin{hyp}\label{newhype}
Concerning $\Delta(G)$, we assume the following:
\begin{enumerate}[(i)]
    \item\label{00} for every vertex in $\rho$, there exists a nonadjacent vertex in $\pi$,
    \item\label{11} for every vertex in $\pi$, there exists a nonadjacent vertex in $\rho$,
    \item\label{33} all the vertices in $\pi$ are admissible. Moreover, no proper connected subgraph with vertex set $\{p\}\cup\pi^*\cup\rho$ occurs as the prime character degree graph of any solvable group,
    \item\label{44} for each vertex set $\pi^*\cup\rho^*$ which induces a complete subgraph in $\Delta(G)$, all the vertices in the corresponding set $\tau$ are admissible. Moreover, no proper connected subgraph with vertex set $\rho(G)\setminus\tau^*$ occurs as the prime character degree graph of any solvable group, and
    \item\label{55} if a disconnected subgraph with vertex set $\rho(G)$ does not occur, then it must specifically violate P\'alfy's inequality from \cite{P2}. Finally, if a disconnected subgraph with vertex set $\rho(G)$ does occur, then the sizes of the connected components must be $n>1$ and $2^n-1$.
\end{enumerate}
\end{hyp}

Observe the following consequences of Hypothesis \ref{newhype}. For \eqref{00} and \eqref{11}, since $\Delta(G)$ is assumed to be connected, it is required that $|\pi|\geq2$ and $|\rho|\geq2$. As another consequence of \eqref{11}, all vertices in $\pi$ are adjacent to each other. Otherwise, if there existed two nonadjacent vertices in $\pi$, then the complement graph would have an odd cycle. This is prohibited by the main theorem from \cite{etal}. Concerning \eqref{33}, one gleans that no proper connected subgraph with vertex set $\rho(G)$ occurs as the prime character degree graph of any solvable group (taking $\pi^*=\pi$). Depending on the initial subset $\pi^*\cup\rho^*$, \eqref{44} may require that some of the vertices in $\rho$ are admissible. Finally, we know from \eqref{55} that if a disconnected subgraph with vertex set $\rho(G)$ does indeed occur as the prime character degree graph of a solvable group, then it must be represented by Example 2.4 from \cite{L}.

\begin{hyp}\label{Ghype}
Concerning the group $G$, we assume $P$ is a normal nonabelian Sylow $p$-subgroup of $G$, and we let $H$ be a $p$-complement for $G$. We set $F:=\mathbf{F}(H)$, which is necessarily nontrivial.
\end{hyp}

Concerning $F$, it is known there exists a degree in $\cd(G)$ that is divisible by all the prime divisors of $|F|$. This corresponds to the vertices in $\pi(|F|)$ inducing a complete subgraph in $\Delta(G)$. Next, assuming $G$ has a normal nonabelian Sylow $p$-subgroup, we can now see the full force of Hypothesis \ref{newhype}\eqref{55}. As stated above, this hypothesis forces the disconnected graph (should it occur) to be represented by Example 2.4 from \cite{Lewis}. In that case, all normal Sylow subgroups are necessarily abelian. Therefore, assuming $N$ is a group such that $\rho(N)=\rho(G)$, we in fact get $\Delta(N)=\Delta(G)$ by Hypothesis \ref{newhype}\eqref{33}. Since $|G|$ is minimal we conclude that $N=G$.

\begin{lemma}\label{lemma00}
Assume Hypotheses \ref{newhype} and \ref{Ghype}. Then $C_H(P)=1$.
\end{lemma}
\begin{proof}
Set $C:=C_H(P)$. We follow the argument on page 259 in \cite{L} in a generalized form to show $C=1$.

First we observe that $C$ is normal in $G$, and that $PC=P\times C$. Therefore, any prime in $\rho(C)$ must be adjacent to $p$ in $\Delta(G)$. So we know that $\rho(C)\subseteq\pi$ and that $r$ divides $|H:C|$ for every $r\in\rho$. Since we know that $H/C$ acts faithfully on $P$, we see that the primes in $\pi(|H:C|)$ must lie in $\rho(G/C)$. Thus, $\rho\subseteq\rho(G/C)$. We then conclude that $\rho(G/C)$ must have one of the following vertex sets: (a) $\{p\}\cup\rho$, (b) $\{p\}\cup\pi^*\cup\rho$, or (c) $\{p\}\cup\pi\cup\rho$. Here $\pi^*$ is taken as a nonempty proper subset of $\pi$.

Suppose (a). Note that none of the primes in $\pi$ divide $|H:C|$. We fix a nonlinear character $\gamma\in\Irr(P)$ and then $\gamma$ extends to $\gamma\times1_C$ in $P\times C$. For $\chi\in\Irr(G|\gamma\times1_C)$ notice that $p$ divides $\chi(1)$ and $\chi(1)/\gamma(1)$ divides $|H:C|$. Since none of the primes in $\rho$ divide $\chi(1)$ and the only primes that could possibly divide $|H:C|$ are the primes in $\rho$, we are forced to conclude that $\chi_{PC}=\gamma\times1_C$, which implies $\chi_P=\gamma$. Using Gallagher's Theorem (Corollary 6.17 from \cite{I}) we have that $\chi(1)d\in\Irr(G)$ for every degree $d\in\cd(H)$. Thus, $\rho(H)\subseteq\pi$. This means that $H$ has a normal abelian Hall $\rho$-subgroup, and hence $O^q(G)<G$ for some $q\in\pi$. This is a contradiction since every vertex in $\pi$ is admissible.

Suppose (b). By Hypothesis \ref{newhype}\eqref{33}, we know that any connected subgraph with vertex set $\{p\}\cup\pi^*\cup\rho$ does not occur as the prime character degree graph of any solvable group. If $\Delta(G/C)$ is a disconnected graph that does occur, then by Theorem 5.5 from \cite{Lewis} we conclude that $G/C$ has a central Sylow $q$-subgroup for some $q\in\pi\setminus\pi^*$. This implies $O^q(G)<G$, which is a contradiction since every vertex in $\pi$ is admissible.

We are left with (c), in which case $\rho(G/C)=\{p\}\cup\pi\cup\rho=\rho(G)$. By Hypothesis \ref{newhype}\eqref{55}, we see that $|G/C|=|G|$, and hence $C=1$.
\end{proof}

Notice that Lemma \ref{lemma00} implies that $H$ acts faithfully on $P$. With this in hand, we now investigate $F$, the Fitting subgroup of $H\cong G/P$.

\begin{claim}\label{rhosub}
Assume Hypotheses \ref{newhype} and \ref{Ghype}. Then $F$ is not a $\rho$-subgroup.
\end{claim}
\begin{proof}
For the sake of contradiction, we assume that $F$ is a $\rho$-subgroup. Let $\gamma\in\Irr(P)$ be a nonlinear character and let $\chi\in\Irr(G|\gamma)$. We see that $p$ divides both $\gamma(1)$ and $\chi(1)$, and that $r$ does not divide $\chi(1)$ or $\chi(1)/\gamma(1)$ for all $r\in\rho$. Since this is true for all characters $\chi\in\Irr(G|\gamma)$, we apply Theorem 12.9 of \cite{MW} to get that $G/P\cong H$ has an abelian Hall $\rho$-subgroup.

Since we assumed that $F$ was in fact a $\rho$-subgroup, we now get that $F$ is the abelian Hall $\rho$-subgroup. Thus $O^q(G)<G$ for some $q\in\pi$, which is a contradiction since all the vertices in $\pi$ are admissible. Therefore $F$ is not a $\rho$-subgroup.
\end{proof}

\begin{claim}\label{pirhosub}
Assume Hypotheses \ref{newhype} and \ref{Ghype}. Then $F$ is not a $\pi^*\cup\rho^*$-subgroup.
\end{claim}
\begin{proof}
For the sake of contradiction, we assume that $F$ is a $\pi^*\cup\rho^*$-subgroup. It is sufficient to suppose $\pi(|F|)=\pi^*\cup\rho^*$. Observe that in this scenario, Hypothesis \ref{newhype}\eqref{00} and \eqref{11} force $\pi^*$ and $\rho^*$ to be proper subsets of $\pi$ and $\rho$, respectively.

We fix the following notation:
$$\eta:=\{s\in\rho(G)\setminus(\{p\}\cup\pi(|F|))~:~s \text{~and~} q \text{~are not adjacent in~} \Delta(G) \text{~for some~} q\in\pi(|F|)\}.$$
In fact, $\eta=\rho(G)\setminus(\{p\}\cup\mathcal{B})$ and $\eta$ must contain at least one element from $\pi\setminus\pi^*$ and from $\rho\setminus\rho^*$. For otherwise this would contradict Hypothesis \ref{newhype}\eqref{00} and \eqref{11} as $\pi^*\cup\rho^*$ forms a complete subgraph of $\Delta(G)$. Thus $\eta$ is nonempty.

Let $\gamma\in\Irr(PF)$ be such that $q$ divides $\gamma(1)$ for all $q\in\pi^*\cup\rho^*$. For any character $\chi\in\Irr(G|\gamma)$, we have that $q$ divides $\chi(1)$ for all $q\in\pi^*\cup\rho^*$, and that $r$ does not divide $\chi(1)$ or $\chi(1)/\gamma(1)$ for any $r\in\eta$. By Theorem 12.9 from \cite{MW}, we see that $G/PF\cong H/F$ has an abelian Hall $\eta$-subgroup. Let $M=O_\mathcal{B}(H)$, and note that $F\subseteq M$. Set $E/M$ to be the Fitting subgroup of $H/M$. Since $H/M$ has an abelian Hall $\eta$-subgroup, we see that $E/M$ is the abelian Hall $\eta$-subgroup. 

Next we have that $\rho(PE)$ contains $\{p\}\cup\pi^*\cup\rho^*\cup\eta$. However this may be a proper subset of $\rho(G)$. Notice that by construction we have that $\rho(G)\setminus(\{p\}\cup\pi^*\cup\rho^*\cup\eta)=\tau$. Hence $\rho(PE)$ must be one of the following: (a) $\rho(G)\setminus\tau^*$ or (b) $\rho(G)$.

Suppose (a). Then by Hypothesis \ref{newhype}\eqref{44} we see that no proper connected subgraph with vertex set $\rho(G)\setminus\tau^*$ occurs as the prime character degree graph of any solvable group. If $\Delta(PE)$ is a disconnected graph that occurs, then by Theorem 5.5 from \cite{Lewis} we observe that $PE$ has a central Sylow $q$-subgroup for some $q\in\tau^*$. This implies $O^q(G)<G$, which is a contradiction since every vertex in $\tau$ is admissible.

We are left with (b). So $\rho(PE)=\rho(G)$, and by Hypothesis \ref{newhype}\eqref{55} we see that $PE=G$. Since $E/M$ is abelian, and since $\eta\cap\pi\neq\varnothing$, we have that $O^s(G)<G$ for some $s\in\pi$. This is a contradiction since all the vertices in $\pi$ are admissible. Therefore $F$ is not a $\pi^*\cup\rho^*$-subgroup.
\end{proof}

\begin{claim}\label{pisub}
Assume Hypotheses \ref{newhype} and \ref{Ghype}. Then $F$ is not a $\pi$-subgroup.
\end{claim}
\begin{proof}
For the sake of contradiction, we suppose $F$ is a $\pi$-subgroup of $H$ and hence $\rho\subseteq\rho(H)$. Furthermore, $H$ acts faithfully on $P$ as consequence of Lemma \ref{lemma00}. We apply Lemma 4.2 from \cite{L} which yields $\rho(H)=\rho$.

Next we will see that $P'$ is minimal normal in $G$. Suppose that $X$ is normal in $G$ such that $1\leq X\leq P'$ and $P'/X$ is a chief factor for $G$. Our goal will be to show that $X=1$. We know $\rho(G/X)\supseteq\rho(G/P')=\rho(G)\setminus\{p\}$ (by Theorem 3 from \cite{MLL}). However, $p\in\rho(G/X)$ since $G/X$ has a normal nonabelian Sylow $p$-subgroup. It follows that $\rho(G)=\rho(G/X)$, and hence $|G/X|=|G|$ by Hypothesis \ref{newhype}\eqref{55}. Therefore $X=1$ and $P'$ is minimal normal in $G$.

The hypotheses of Lemma 4.3 from \cite{L} are satisfied, and therefore $\Delta(P'H)$ has disconnected components with vertex sets $\rho$ and $\pi^*$ such that
\begin{equation}\label{cool}
|\pi^*|\geq2^{|\rho|}-1.
\end{equation}
In particular, notice that
\begin{equation}\label{good}
|\rho|\leq|\pi^*|.
\end{equation}

Hypothesis \ref{newhype}\eqref{55} gives two cases concerning the possibility of a disconnected subgraph of $\Delta(G)$ with vertex set $\rho(G)$. We will investigate these below, but first notice that by our assumptions, the two connected components of a possibly disconnected subgraph of $\Delta(G)$ with vertex set $\rho(G)$ are forced to have vertex sets $\rho$ and $\{p\}\cup\pi$. 

For Case 1, we consider if no disconnected subgraph of $\Delta(G)$ with vertex set $\rho(G)$ occurs. By Hypothesis \ref{newhype}\eqref{55} we know that for component sizes $a$ and $b$ (with $a\leq b$), we must have the situation where $b<2^a-1$. This gives two subcases:

Case 1(a): $|\rho|=a$ and $|\{p\}\cup\pi|=b$. For this scenario we have
$$|\pi^*|\leq|\pi|<|\{p\}\cup\pi|=b<2^a-1=2^{|\rho|}-1,$$
which contradicts \eqref{cool}.

Case 1(b): $|\{p\}\cup\pi|=a$ and $|\rho|=b$. Thus
$$|\rho|=b\geq a=|\{p\}\cup\pi|>|\pi|\geq|\pi^*|,$$
which contradicts \eqref{good}.

For Case 2, we consider if a disconnected subgraph of $\Delta(G)$ with vertex set $\rho(G)$ occurs. By Hypothesis \ref{newhype}\eqref{55} we know the components must be of sizes $n>1$ and $2^n-1$, which gives two subcases:

Case 2(a): $|\rho|=n>1$ and $|\{p\}\cup\pi|=2^n-1$. Observe that
$$|\pi^*|\leq|\pi|<|\{p\}\cup\pi|=2^n-1=2^{|\rho|}-1,$$
which contradicts \eqref{cool}.

Case 2(b): $|\{p\}\cup\pi|=n>1$ and $|\rho|=2^n-1$. Similarly notice
$$|\rho|=2^n-1>n=|\{p\}\cup\pi|>|\pi|\geq|\pi^*|,$$
which contradicts \eqref{good}. Therefore $F$ is not a $\pi$-subgroup.
\end{proof}

We have shown that no primes in $\rho(G)\setminus\{p\}$ will divide $|F|$. This leads to the following result:

\begin{theorem}\label{newmain}
Assume Hypothesis \ref{newhype}. Then $G$ has no normal nonabelian Sylow $p$-subgroup.
\end{theorem}
\begin{proof}
For the sake of contradiction, suppose Hypothesis \ref{Ghype}. That is, suppose $G$ has a normal nonabelian Sylow $p$-subgroup $P$, and let $H$ be a $p$-complement for $G$ with Fitting subgroup $F$.

By Claims \ref{rhosub}, \ref{pirhosub}, and \ref{pisub}, we see that $F$ must be trivial since no primes divide $|F|$. As $H$ is solvable it must have a nontrivial Fitting subgroup, a contradiction.
\end{proof}

Consider the following graphs with the appropriate vertex $p$ labeled in Figure \ref{fig} for several examples and applications of Theorem \ref{newmain}.
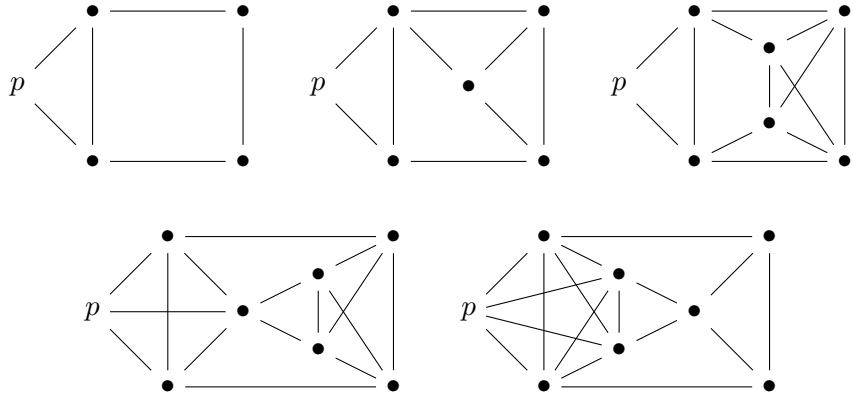
\begin{figure}[H]
    \centering
$
\begin{tikzpicture}[scale=2]
\node (1a) at (1.5,1) {$\bullet$};
\node (2a) at (1.5,0) {$\bullet$};
\node (3a) at (.5,1) {$\bullet$};
\node (4a) at (.5,0) {$\bullet$};
\node (5a) at (0,.5) {$p$};
\path[font=\small,>=angle 90]
(1a) edge node [right] {$ $} (2a)
(1a) edge node [above] {$ $} (3a)
(2a) edge node [above] {$ $} (4a)
(3a) edge node [above] {$ $} (4a)
(3a) edge node [right] {$ $} (5a)
(4a) edge node [above] {$ $} (5a);
\node (1b) at (2.5,1) {$\bullet$};
\node (2b) at (2,.5) {$p$};
\node (3b) at (3.5,1) {$\bullet$};
\node (4b) at (2.5,0) {$\bullet$};
\node (5b) at (3,.5) {$\bullet$};
\node (6b) at (3.5,0) {$\bullet$};
\path[font=\small,>=angle 90]
(1b) edge node [right] {$ $} (2b)
(1b) edge node [above] {$ $} (3b)
(1b) edge node [above] {$ $} (4b)
(2b) edge node [above] {$ $} (4b)
(3b) edge node [right] {$ $} (5b)
(3b) edge node [above] {$ $} (6b)
(1b) edge node [above] {$ $} (5b)
(4b) edge node [above] {$ $} (6b)
(5b) edge node [above] {$ $} (6b);
\node (1c) at (4,.5) {$p$};
\node (2c) at (4.5,1) {$\bullet$};
\node (3c) at (4.5,0) {$\bullet$};
\node (4c) at (5,.75) {$\bullet$};
\node (5c) at (5,.25) {$\bullet$};
\node (6c) at (5.5,1) {$\bullet$};
\node (7c) at (5.5,0) {$\bullet$};
\path[font=\small,>=angle 90]
(1c) edge node [right] {$ $} (2c)
(1c) edge node [above] {$ $} (3c)
(2c) edge node [above] {$ $} (3c)
(2c) edge node [above] {$ $} (4c)
(2c) edge node [right] {$ $} (6c)
(3c) edge node [above] {$ $} (5c)
(3c) edge node [right] {$ $} (7c)
(4c) edge node [above] {$ $} (5c)
(4c) edge node [above] {$ $} (6c)
(4c) edge node [above] {$ $} (7c)
(5c) edge node [above] {$ $} (6c)
(5c) edge node [above] {$ $} (7c)
(6c) edge node [above] {$ $} (7c);
\node (1y) at (.5,-1) {$p$};
\node (2y) at (1,-.5) {$\bullet$};
\node (3y) at (1,-1.5) {$\bullet$};
\node (4y) at (1.5,-1) {$\bullet$};
\node (5y) at (2,-.75) {$\bullet$};
\node (6y) at (2,-1.25) {$\bullet$};
\node (7y) at (2.5,-.5) {$\bullet$};
\node (8y) at (2.5,-1.5) {$\bullet$};
\path[font=\small,>=angle 90]
(1y) edge node [right] {$ $} (2y)
(1y) edge node [above] {$ $} (3y)
(1y) edge node [above] {$ $} (4y)
(2y) edge node [above] {$ $} (3y)
(2y) edge node [above] {$ $} (4y)
(2y) edge node [right] {$ $} (7y)
(3y) edge node [above] {$ $} (4y)
(3y) edge node [right] {$ $} (8y)
(4y) edge node [above] {$ $} (5y)
(4y) edge node [above] {$ $} (6y)
(5y) edge node [above] {$ $} (6y)
(5y) edge node [above] {$ $} (7y)
(5y) edge node [above] {$ $} (8y)
(6y) edge node [above] {$ $} (7y)
(6y) edge node [above] {$ $} (8y)
(7y) edge node [above] {$ $} (8y);
\node (1z) at (3,-1) {$p$};
\node (2z) at (3.5,-.5) {$\bullet$};
\node (3z) at (3.5,-1.5) {$\bullet$};
\node (4z) at (4,-.75) {$\bullet$};
\node (5z) at (4,-1.25) {$\bullet$};
\node (6z) at (4.5,-1) {$\bullet$};
\node (7z) at (5,-.5) {$\bullet$};
\node (8z) at (5,-1.5) {$\bullet$};
\path[font=\small,>=angle 90]
(1z) edge node [right] {$ $} (2z)
(1z) edge node [above] {$ $} (3z)
(1z) edge node [above] {$ $} (4z)
(1z) edge node [above] {$ $} (5z)
(2z) edge node [above] {$ $} (3z)
(2z) edge node [right] {$ $} (4z)
(2z) edge node [above] {$ $} (5z)
(2z) edge node [right] {$ $} (7z)
(3z) edge node [above] {$ $} (4z)
(3z) edge node [above] {$ $} (5z)
(3z) edge node [above] {$ $} (8z)
(4z) edge node [above] {$ $} (5z)
(4z) edge node [above] {$ $} (6z)
(5z) edge node [above] {$ $} (6z)
(6z) edge node [above] {$ $} (7z)
(6z) edge node [above] {$ $} (8z)
(7z) edge node [above] {$ $} (8z);
\end{tikzpicture}
$
    \caption{Graphs satisfying  Hypothesis \ref{newhype}}
    \label{fig}
\end{figure}

\end{document}